\documentclass[a4paper,10pt]{amsart}
\usepackage{amsthm}
\usepackage{latexsym}
\usepackage{amssymb}
\usepackage[all]{xy}
\newtheorem{theorem}{Theorem}[section]
\newtheorem*{nonbtheorem}{Theorem}
\newtheorem{lemma}[theorem]{Lemma}
\newtheorem{proposition}[theorem]{Proposition}
\newtheorem{corollary}[theorem]{Corollary}
\theoremstyle{definition}

\theoremstyle{remark}
\newtheorem{remark}[theorem]{Remark}

\newcommand{\Mod}{\text{\rm mod-}}
\newcommand{\Hom}{\operatorname{Hom}}

\newcommand{\Ext}{\operatorname{Ext}}
\newcommand{\End}{\operatorname{End}}
\newcommand{\Ker}{\operatorname{Ker}}
\newcommand{\Ima}{\operatorname{Im}}

\newcommand{\rad}{\operatorname{rad}}

\newcommand{\Sing}{\operatorname{Sing}}
\newcommand{\Irr}{\operatorname{Irr}}
\newcommand{\corank}{\operatorname{corank}}

\title{Ringel-Hall numbers and the Gabriel-Roiter submodules of simple homogeneous regular modules}
\author{Csaba Sz\'ant\'o, Istv\'an Sz\"oll\H osi}

\begin{document}
\maketitle \pagestyle{myheadings} \markboth{\sc Csaba
Sz\'ant\'o, Istv\'an Sz\"oll\H osi}{\sc Ringel-Hall numbers and the Gabriel-Roiter submodules of simple homogeneous regular modules}

{\bf Abstract.} Let $k$ be an arbitrary field and $Q$ an acyclic quiver of tame type (i.e. of type $\tilde
A_n,\tilde D_n,\tilde E_6, \tilde E_7,\tilde E_8$). Consider the path algebra $kQ$, the category of finite dimensional right modules $\Mod kQ$ and denote by $\delta$ the minimal radical vector of $Q$. In the first part of the paper we deduce that the Gabriel-Roiter inclusions in preprojective indecomposables and simple homogeneous regulars of dimension $\delta$ as well as their Gabriel-Roiter measures are field independent. Using this result we can prove in a more general setting a theorem by Bo Chen which states that the GR submodule of a simple homogeneous regular of dimension $\delta$ is a preprojective of defect $-1$. The generalization consists in considering the originally missing case $\tilde E_8$ and using arbitrary fields (instead of algebraically closed ones). Our proof is based on the idea of Ringel (used in the Dynkin quiver context) of comparing all possible Ringel-Hall polynomials with the special form they take in case of a Gabriel-Roiter inclusion. For this purpose we compute (using a program written in GAP) a list of useful Ringel-Hall polynomials in tame cases.

{\bf Key words.}  Tame hereditary algebra, Gabriel-Roiter measure, Gabriel-Roiter inclusion, Ringel-Hall algebra,
Ringel-Hall polynomials.

\medskip

{\bf 2000 Mathematics Subject Classification.} 16G20.

\section{\bf Introduction}
Let $k$ be an arbitrary field and consider the path
algebra $kQ$ where $Q$ is a quiver of tame type i.e. of type $\tilde
A_n,\tilde D_n,\tilde E_6, \tilde E_7,\tilde E_8$. When $Q$ is of
type $\tilde A_n$ we exclude the cyclic orientation, so $kQ$ is a
finite dimensional tame hereditary algebra. We denote by $\Mod
kQ$ the category of finite dimensional (hence finite) right modules. Let $[M]$ be the isomorphism class of $M\in\Mod kQ$ and $tM=M\oplus\dots\oplus M$ ($t$-times).
The category $\Mod kQ$ can and will be
identified with the category $\text{\rm rep-} kQ$ of the finite
dimensional $k$-representations of the quiver
$Q=(Q_0=\{1,2,...,n\},Q_1)$. Here $Q_0=\{1,2,...,n\}$ denotes the
set of vertices of the quiver and $Q_1$ the set of arrows. For an
arrow $\alpha$ we denote by $s(\alpha)$ the starting point of the
arrow and by $e(\alpha)$ its endpoint. Recall that a
$k$-representation of $Q$ is defined as a set of finite dimensional
$k$-spaces $\{V_i|i=\overline{1,n}\}$ corresponding to the vertices
together with $k$-linear maps $V_{\alpha}:V_{s(\alpha)}\to
V_{e(\alpha)}$ corresponding to the arrows. The dimension of a
module $M=(V_i,V_{\alpha})\in\Mod kQ=\text{\rm rep-} kQ$ is $\underline\dim M=(\dim_kV_i)_{i=\overline{1,n}}\in\Bbb Z^n$. For
$x=(x_i),y=(y_i)\in\Bbb Z^n$ we say that $x\leq y$ iff $x_i\leq y_i$ $\forall i$.

Let $P(i)$ (resp. $I(i)$) the projective (resp. injective) indecomposable
corresponding to the vertex $i$ and consider the Cartan matrix $C_Q$
with the $j$-th column being $\underline\dim P(j)$. We have a
biliniar form on $\Bbb Z^n$ defined as $\langle
a,b\rangle=aC_Q^{-t}b^t$. Then for two modules $X,Y\in\Mod kQ$ we
get
$$\langle\underline\dim X,\underline\dim Y\rangle=\dim_k\Hom
(X,Y)-\dim_k\Ext^1(X,Y).$$ We denote by $\bf{q}$ the quadratic form
defined by ${\bf q}(a)=\langle a,a\rangle$. This form is called Euler quadratic form of $kQ$ and it depends only on the underlying graph of $Q$ (so it does not depend on $k$ and on the orientation of $Q$). It is also positive semi-definite with radical $\Bbb Z\delta$, i.e.
$\{x\in\Bbb Z^n|{\bf q}(x)=0\}=\Bbb Z\delta$. The vector $\delta$ is known
for each type $\tilde A_n,\tilde D_n,\tilde E_6, \tilde E_7,\tilde
E_8$ (see \cite{dlabrin}). A vector $x\in\Bbb N^n$ is called
positive real root of ${\bf q}$ if ${\bf q}(x)=1$. It is known (see
\cite{dlabrin}) that for all positive real roots $x$ there is a
unique (up to isomorphism) indecomposable module $M\in\Mod kQ$ with
$\underline\dim M=x$. The rest of the indecomposables are of
dimension $t\delta$ with $t$ positive integers. The defect of a
module $M$ is $\partial M=\langle\delta,\underline\dim
M\rangle=-\langle\underline\dim M,\delta\rangle$. For a short exact
sequence $0\to X\to Y\to Z\to 0$ we have that $\partial Y=\partial
X+\partial Z$. For $x\in\Bbb N^n$ we define its defect $\partial x=\langle\delta,x\rangle=-\langle x,\delta\rangle$.

Consider the Auslander-Reiten translates $\tau=D\Ext^1(-,kQ)$ and
$\tau^{-1}=\Ext^1(D(kQ),-)$, where $D=\Hom_k(-,k)$. An indecomposable module $M$ is preprojective (preinjective) if
there exists a positive integer $m$ such that $\tau^m(M)=0$
($\tau^{-m}(M)=0$). Otherwise $M$ is said to be regular. Note that
the dimension vectors of preprojective and preinjective
indecomposables are positive real roots of ${\bf q}$ which determine uniquely (up to isomorphism) these indecomposables.
A module is preprojective (preinjective, regular) if every indecomposable
component is preprojective (preinjective, regular). Note that an
indecomposable module $M$ is preprojective (preinjective, regular)
iff $\partial M<0$ ($\partial M>0$, $\partial M=0$). Moreover if $Q$
is of type $\tilde A_n$ then $\partial M=-1$ for $M$ preprojective
indecomposable and $\partial M=1$ for $M$ preinjective
indecomposable. For a positive root $x$ with $\partial x<0$ we will denote the corresponding unique preprojective indecomposable by $P(x,k)$.

The vertices of the Auslander-Reiten quiver associatiated to $kQ$ are the isomorphism
classes of indecomposables, while its arrows correspond to so called
irreducible maps. Among its components it will have a preprojective component (with all
the isoclasses of indecomposable preprojectives) and a preinjective
component (with all the isoclasses of indecomposable preinjectives).
All the other components (containing the isoclasses of
indecomposable regulars) are "tubes" of the form $\Bbb
ZA_{\infty}/m$, where $m$ is the rank of the tube. The tubes are
indexed by the points of the projective line $\Bbb P_k^1$, the degree of a
point $a\in\Bbb P^1_k$ being denoted by $\deg a$. A tube of rank 1
is called homogeneous, otherwise it is called non-homogeneous. We have
at most 3 non-homogeneous tubes indexed by points $a$ of degree
$\deg a=1$. All the other tubes are homogeneous. We assume
that the non-homogeneous tubes are labelled by some subset of $\{0,1,\infty\}$, whereas
the homogeneous tubes are labelled by the closed points of the scheme $\mathbb H_k=\mathbb H_{\mathbb Z}\otimes k$ for some
open integral subscheme $\mathbb H_{\mathbb Z}\subset \mathbb P_{\mathbb Z}^1$. A regular indecomposable taken from a homogeneous tube labelled by a point $a\in\mathbb H_k$ is called homogeneous regular. If its dimension is $(\deg a)\delta$, we call it simple homogeneous regular (it is placed on the mouth of the tube) and denote it by $R^k_{\deg a} (1,a)$. Note that in case $Q$ is not of type $\tilde A_n$ and $k$ is finite with $q$ elements the number of points $a\in\mathbb H_k$ with degree 1 is $q-2$, so in order to have homogeneous tubes, $k$ must have at least 3 elements.

We recall next the definition of the Gabriel-Roiter measure, Gabriel-Roiter submodule, Gabriel-Roiter inclusion and Gabriel-Roiter exact sequence (see \cite{rin2},\cite{rin4},\cite{rin5},\cite{chen1},\cite{chen2}).
For $M\in\Mod kQ$ we denote by $|M|$ the length of $M$. For a vector $x\in\mathbb N^n$ we have $|x|=\sum_{i=1}^nx_i$. Let $\mathcal{P}(\Bbb N^*)$ be the set of
all subsets of $\Bbb N^*$. A total order on $\mathcal{P}(\Bbb N^*)$ can
be defined as follows: if $I,J$ are two different subsets of $\Bbb
N$, write $I < J$ if the smallest element in $(I\setminus
J)\cup(J\setminus I)$ belongs to $J$. Using the total order above, for each $M\in\Mod kQ$ let
$\mu(M)$ be the maximum of the sets $\{|M_1|, |M_2|,..., |M_t|\}$,
where $M_1\subset M_2\subset...\subset M_t$ is a chain of
indecomposable submodules of $M$. We call $\mu(M)$ the
Gabriel-Roiter (GR for short) measure of $M$. If $M\in\Mod kQ$ is an
indecomposable module, an indecomposable submodule $X\subset M$ is called a GR submodule provided $\mu(M) = \mu(X)\cup\{|M|\}$,
thus if and only if every proper submodule of $M$ has GR measure at
most $\mu(X)$. A monomorphism $N\to M$ between two indecomposable modules is called GR inclusion if $\mu(M)=\mu(N)\cup\{|M|\}$ (i.e. $N$ is isomorphic with a GR submodule of $M$). It is known that the factor of a GR inclusion is indecomposable. If $N\subset M$
is a GR inclusion the exact sequence $0\to N\to M\to M/N\to 0$ will
be called a GR exact sequence.

Ringel proves in \cite{rin2} that in case $Q$ is a Dynkin quiver, the GR inclusions in
indecomposables and their GR measures are independent from the base field $k$.

We prove in this paper that in case $Q$ is tame without oriented cycles, the GR inclusions in preprojective indecomposables and simple homogeneous regulars of dimension $\delta$ as well as their GR measures are also field independent. More precisely we prove the following two theorems:

\begin{nonbtheorem} Consider two fields $k,k'$. Then for preprojective indecomposables we have $\mu(P(x,k))=\mu(P(x,k'))$, moreover we have a GR inclusion
$P(x,k)\to P(y,k)$ iff we have a GR inclusion $P(x,k')\to P(y,k')$.
\end{nonbtheorem}

\begin{nonbtheorem}Consider two fields $k,k'$ with at least 3 elements and $a\in\mathbb H_{k}$, $a'\in\mathbb H_{k}'$ points of degree 1. Then for simple homogeneous regulars of dimension $\delta$ we have  $\mu(R^k_1(1,a))=\mu(R^{k'}_1(1,a'))$ and we have a GR inclusion
$P(x,k)\to R^k_1(1,a)$ iff we have a GR inclusion $P(x,k')\to R^{k'}_1(1,a')$.
\end{nonbtheorem}

In order to prove these theorems we use Ringel's ideas from \cite{rin2}, Bo Chen's results from \cite{chen2} but we also introduce some new concepts: Ringel-Hall polynomials in tame cases (see \cite{hubery}) and results from algebraic geometry.

As an application of the theorems above one can prove a result by Bo Chen in \cite{chen1} in a more general context: our result is valid also for the case $\tilde E_8$ (this case is missing from \cite{chen1}) and it is field independent (in \cite{chen1} $k$ is algebraically closed).

\begin{nonbtheorem} Let $Q$ be not of type $\tilde A_n$ and $k$ a field having at least 3 elements. Consider a simple homogeneous regular module $R$ of dimension $\delta$.
If there is a GR inclusion $P\to R$ then $\partial P=-1$. (In case $Q$ is of type $\tilde A_n$ then every preprojective indecomposable has defect -1).
As a consequence for every tame quiver $Q$, the pair $(R/P,P)$ is a Kronecker pair.
\end{nonbtheorem}

The proof of the theorem above follows the idea of Ringel from \cite{rin2}: one compares all possible Ringel-Hall polynomials with the special form they take in case of a GR inclusion. The problem is that although we know these polynomials in the Dynkin case (see \cite{rin6}), in the tame case little is known in this sense.  So we have to determine in the tame cases a list of Ringel-Hall polynomials corresponding to exact sequences of the form $0\to P\to R\to I\to 0$, where $P$ is a preprojective, $I$ a preinjective indecomposable and $R$ is a simple homogeneous regular of dimension $\delta$. Using reflection functors we can reduce the problem to the case when $P$ is simple and $Q$ has a unique sink. In this case the polynomials above are computed with the help of a computer program written in GAP (see \cite{GAP}). The computation is very quick, it takes about 15 minutes. The list of these special Ringel-Hall polynomials may have other interesting applications in representation theory.

\section{\bf Facts on tame hereditary algebras and associated Ringel-Hall algebras}
For a detailed description of the forthcoming notions we refer to
\cite{assem},\cite{aus},\cite{dlabrin},\cite{rin1}. First of all we fix some notations and conventions.

Consider a representation $M$ of dimension $x$ such that its linear application $k^{x_{s(\alpha)}}\to k^{x_{e(\alpha)}}$ corresponding to the arrow $\alpha$ is given by the matrix $A_{\alpha}$ (in the canonical base). An endomorphism of this representation using the canonical bases can be identified with a collection $(X_i)_{i\in Q_0}$ of square matrices $X_i$ of dimension $x_i$ which satisfy the relations $A_{\alpha}X_{s(\alpha)}=X_{e(\alpha)}A_{\alpha}$. These relations induce a homogeneous linear system of equations with the unknowns being the elements of $X_i$. Denote by $A_M$ the matrix of this system. Trivially we have $\dim_k\End(M)=\corank A_M$.

Let $k$ be a field with prime field $k_0$, $x$ a positive real root such that $\partial x<0$ and $P(x,k)$ the (unique up to isomorphism) preprojective indecomposable representation over $k$ with dimension $x$. In case we use the representation $P(x,k_0)$ over the prime field $k_0$ of $k$, we agree that the representation
$P(x,k)$ is constructed in the following way: for $\alpha\in Q_1$ consider the linear application $k^{x_{s(\alpha)}}\to k^{x_{e(\alpha)}}$ having the same matrix in the canonical base as the linear application $k_0^{x_{s(\alpha)}}\to k_0^{x_{e(\alpha)}}$ in the representation $P(x,k_0)$. Note that $P(x,k)$ is an indecomposable representation over $k$ of dimension $x$ (since $A_{P(x,k)}=A_{P(x,k_0)}$, so $\dim_k\End(P(x,k))=\corank A_{P(x,k)}=\corank A_{P(x,k_0)}=1=\dim_k\End(P(x,k_0))$). A second convention is that if we use the rational representation $P(x,\mathbb Q)$ with rational matrices corresponding to the arrows (in the canonical base) $A_\alpha$, then for a big enough prime $p$ the representation $P(x,\mathbb F_p)$ has matrices $A_{\alpha}\mod p$. Note also that this representation $P(x,\mathbb F_p)$ is indecomposable for $p$ big enough.

We will describe now the so called decomposition symbol used by Hubery in \cite{hubery}.

We know that the isomorphism class of a module in $\Mod kQ$ without homogeneous regular summands can be described
combinatorially, whereas homogeneous regular modules are placed on homogeneous tubes labelled by the closed points of the scheme $\mathbb H_k$. Since indecomposables from different tubes have no nonzero homomorphisms
and no non-trivial extensions and any indecomposable regular module is regular uniserial, we have that the homogeneous regular indecomposables on the tube labelled by $a\in\mathbb H_k$ with $\deg a=d$ are $R^k_d(1,a)\subset R^k_d(2,a)\subset R^k_d(3,a)\subset\dots$ all the factors being isomorphic with $R^k_d(1,a)$. For a partition $\lambda=(\lambda_1,...,\lambda_n)$ let $R^k_d(\lambda,a)=R^k_d(\lambda_1,a)\oplus\dots\oplus R^k_d(\lambda_n,a)$. A
decomposition symbol is a pair $\alpha=(\mu,\sigma)$ such that $\mu$ specifies (uniquely up to isomorphism) a module without
homogeneous regular summands and $\sigma=\{(\lambda^1, d_1),\dots,(\lambda^r, d_r)\}$ is a Segre symbol.
Given a decomposition symbol $\alpha=(\mu,\sigma)$ and a field $k$, we define the decomposition class $S(\alpha,k)$ to be
the set of isomorphism classes of modules of the form $M(\mu,k)\oplus R$, where $M(\mu,k)$
is the $kQ$-module determined by $\mu$ and $R=R^k_{d_1}(\lambda^1,a_1)\oplus\dots\oplus R^k_{d_r}(\lambda^r,a_r)$
for some distinct points $a_1,\dots a_r\in\mathbb H_k$ such that $\deg a_i = d_i$.

Note that for $k$ finite with $q$ elements the number of points $a\in\mathbb H_k$ of degree 1 is $q+1$, $q$ or $q-1$ in the $\tilde A_n$ case (depending on the quiver) and  $q-2$ for tame quivers of type $\tilde D_n,\tilde E_6, \tilde E_7,\tilde E_8$. The number of points $a\in\mathbb H_k$ of degree $l\geq 2$ is $N(q,l)=\frac1{d}\sum_{d|l}\mu(\frac{l}d)q^d$, where $\mu$ is the M\"obius function and $N(q,l)$ is the
number of monic, irreducible polynomials of degree $l$ over a field with $q$ elements. We can conclude that for a decomposition symbol $\alpha$ the polynomial  $n_{\alpha}(q)=|S(\alpha,k)|$ is strictly increasing in $q>1$ (see \cite{hubery}).

Finally we fix that for a preprojective (preinjective) indecomposable corresponding to the positive real root $x$ the decomposition symbol has the form $\alpha=(x,\emptyset)$ and for a simple homogeneous regular of dimension $\delta$ the decomposition symbol is $(\emptyset,((1),1))$.

The following lemma summarizes some well known properties in the category $\Mod kQ$.
\begin{lemma} {\rm a)} For $P$ preprojective, $I$ preinjective and $R$ regular module
we have $\Hom ({R},{P})=\Hom ({I},{P})=\Hom
({I},{R})=\Ext^1({P},{R})=\Ext^1({P},{I})= \Ext^1({R},{I})=0.$ It follows that the submodules of a preprojective are always preprojective and a submodule of a regular cannot have preinjective components.

{\rm b)} Indecomposables from different tubes have no nonzero homomorphisms
and no non-trivial extensions.

{\rm c)} For $P$ indecomposable preprojective, $I$ indecomposable
preinjective and $R^k_1(1,a)$ simple homogeneous regular of dimension $\delta$ we have
$$\dim_k\End(P)=\dim_k\End(I)=\dim_k\End(R^k_1(1,a))=1.$$
\end{lemma}

Suppose now that $k$ is finite. We consider the rational Ringel-Hall algebra $\mathcal{H}(kQ)$
of the algebra $kQ$. Its $\Bbb Q$-basis is formed by the isomorphism
classes $[M]$ from $\Mod kQ$ and the multiplication is defined by
$$[N_1][N_2]=\sum_{[M]}F^M_{N_1N_2}[M].$$ The structure constants
$F^M_{N_1N_2}=|\{U\subseteq M|\ U\cong N_2,\ M/U\cong N_1\}|$ are
called Ringel-Hall numbers.

There is an important theorem by Hubery in \cite{hubery} which proves the existence of Ringel-Hall polynomials in tame cases with respect to the decomposition classes. More precisely we have the following theorem
\begin{theorem}{\rm(\cite{hubery})} Given decomposition symbols $\alpha,\beta$ and $\gamma$, there exists a
rational polynomial $F^{\gamma}_{\alpha\beta}$ such that for any finite field $k$ with $|k|=q$,
$$F^{\gamma}_{\alpha\beta}(q)=\sum_{\begin{array}{c}A\in S(\alpha,k)\\B\in S(\beta,k)\end{array}}F^C_{AB}\text{\hskip1cm for all }C\in S(\gamma,k).$$
\end{theorem}

\section{\bf Facts on GR measure and GR inclusions}
In this section we collect some well-known properties of GR measures and inclusions (see \cite{rin2},\cite{rin4},\cite{rin5},\cite{chen1},\cite{chen2} for details).

\begin{lemma} Let $X$, $Y$ and $Z$ be indecomposable modules.

{\rm a)} If $X$ is a proper submodule of $Y$, then $\mu(X)<\mu(Y)$.

{\rm b)} If $\mu(X)<\mu(Y)<\mu(Z$) and there is a GR inclusion $X\to Z$,
then $|Y|>|Z|$.
\end{lemma}

For two sets $I,J\in\mathcal{P}(\Bbb N^*)$ we say that $J$ starts with $I$ provided $I=J$ or $I\subset J$ and for all elements $a\in I$ and $b\in J\setminus I$ we have $a<b$.

\begin{theorem}{\rm (Third part of Main Property (Gabriel, Ringel),\cite{rin4})}. Let $X,Y_1,\dots Y_t$ be indecomposable modules and assume
that there is a monomorphism $f:X\to Y_1\oplus\dots\oplus Y_t$. If $\max\{\mu(Y_i)\}$ starts with $\mu(X)$ then there is some $j$ such that $\pi_j f$ is injective, where $\pi_j:Y_1\oplus\dots\oplus Y_t\to Y_j$ is the canonical projection.
\end{theorem}

For $X,Y$ indecomposables let us denote by $\Sing(X,Y)$ the set of maps $X\to Y$ which are not monomorphisms.
\begin{proposition}{\cite{rin2}} Suppose  $Y$ is not simple and $X\to Y$ is a GR inclusion. We have the following:

{\rm a)} $\Sing(X,Y)$ is a $k$-subspace of $\Hom(X,Y)$.

{\rm b)} Suppose $k$ is finite with $q$ elements and let $e=\dim_k\End(X)$, $r=\dim_k\rad\End(X)$,
$h=\dim_k\Hom(X,Y)$, $s=\dim_k\Sing(X,Y)$.
Then the number $u^Y_X$ of submodules of $Y$ which are isomorphic to $X$ is
$\frac{q^{s-r}(q^{h-s}-1)}{(q^{e-r}-1)}$ (and $h>s\geq r,e>r$).
\end{proposition}

\section{\bf GR measure of preprojectives}
The aim of this section is to prove that the GR inclusions in
preprojectives are field independent. In particular the GR measure
of preprojectives is also field independent. We should keep in mind
that submodules of preprojectives are preprojective and a preprojective indecomposable of dimension $x$ is denoted by $P(x,k)$, where $x$ is a positive real root with $\partial x<0$.

We know due to Ringel (see \cite{rin2}) that the previous statement
is true for all the indecomposables in the Dynkin case. Our approach
for the tame case uses Ringel's ideas from \cite{rin2} but also
introduces some new concepts. Note that contrary to the Dynkin case, we cannot use Schofield short exact sequences.

Our first lemma is a straightforward generalization of the corresponding result by Ringel in the Dynkin case (see \cite{rin2}).

\begin{lemma} If there is a GR inclusion $P(x,k)\to P(y,k)$ then there is a monomorphism $P(x,k_0)\to P(y,k_0)$.
\end{lemma}
\begin{proof} Suppose that for two indecomposable preprojectives $P(z,k)$ and $P(t,k)$ we have an irreducible morphism $P(z,k)\to P(t,k)$. Knowing the
structure of the AR quiver (no multiple arrows) it follows that
$\dim_k\Irr(P(z,k),P(t,k))=1$, but this implies also that
$\dim_{k_0}\Irr(P(z,k_0),P(t,k_0))=1$. Denote by $f^0_{z,t}$ an
irreducible morphism $P(z,k_0)\to P(t,k_0)$, so its class modulo
$\rad^2(P(z,k_0),P(t,k_0))$ is a nonzero element from
$\Irr(P(z,k_0),P(t,k_0))$. Viewing $f^0_{z,t}$ as a collection of
square matrices over $k_0$ corresponding to the vertices (satisfying
some matrix equations), $f^0_{z,t}$ can be also considered as a
morphism $f_{z,t}:P(z,k)\to P(t,k)$ (having in mind the conventions
made in the begining of Section 2). Moreover we have that $f_{z,t}$ is also irreducible. Indeed,
there is an almost split sequence over $k_0$ of the form
$$\xymatrix{0\ar[r]& P(z,k_0)\ar[r]^{\tiny\left(\begin{array}{c} f^0_{z,t}\\f^0
\end{array}\right)}& P(t,k_0)\oplus P'(k_0)\ar[r] & \tau^{-1}P(z,k_0)\ar[r]&0}.$$
Keeping all the matrices over $k_0$, the exact sequence above has a natural form over $k$
$$\xymatrix{0\ar[r]& P(z,k)\ar[r]^{\tiny\left(\begin{array}{c} f_{z,t}\\f
\end{array}\right)}& P(t,k)\oplus P'(k)\ar[r] & \tau^{-1}P(z,k)\ar[r]&0},$$ which remains almost split (since $\overline{\End}(P(z,k))\cong k$, see \cite{aus} V.2).
So we can conclude that the class of $f_{z,t}$ forms a basis in $\Irr(P(z,k),P(t,k))$. From now on we fix all these special irreducible morphisms.

As in the Dynkin case, it is true for the preprojectives that every element
in $\Hom(P(x,k),P(y,k))$ is a linear combination of compositions of
irreducible basis maps of type $f_{z,t}$ (see
\cite{aus} VIII.1, V.7,\cite{assem} VIII.2, IV.5).

Since $P(x,k)\to P(y,k)$ is a GR inclusion then we know by Ringel
(see \cite{rin2}) that $\Sing(P(x,k),P(y,k))$ is a $k$-subspace of
$\Hom(P(x,k),P(y,k))$. So if all nonzero compositions of irreducible
basis maps are in $\Sing(P(x,k),P(y,k))$ then their linear
combination is also in $\Sing(P(x,k),P(y,k))$, thus
$\Sing(P(x,k),P(y,k))=\Hom(P(x,k),P(y,k))$ which is impossible. This
means that there is a nonzero composition of irreducible basis maps
which is not in $\Sing(P(x,k),P(y,k))$, so it is mono. But all the
basis maps $f_{z,t}$ are defined over $k_0$, so the composition above
will give us the desired monomorphism $P(x,k_0)\to P(y,k_0)$.
\end{proof}

Using the convention from the beginning of Section 2, the following lemma is trivial:
\begin{lemma} If there is a monomorphism $P(x,k_0)\to P(y,k_0)$ then there is a monomorphism $P(x,k)\to P(y,k)$.
\end{lemma}

Using the previous two lemmas we have the following corollary:
\begin{corollary} Consider a field $k$ and its prime field $k_0$. Then for GR
measures we have $\mu(P(x,k_0))=\mu(P(x,k))$, moreover we have a
GR inclusion $P(x,k_0)\to P(y,k_0)$ iff we have a GR inclusion
$P(x,k)\to P(y,k)$.
\end{corollary}
\begin{proof} If $\mu(P(x,k))=\{n_1,...,n_t\}$ then (using the fact that submodules
of preprojectives are preprojective) there is a sequence of GR
inclusions $P(x_1,k)\to...\to P(x_t,k)=P(x,k)$ with
$|P(x_i,k)|=n_i$. By Lemma 4.1 there is a chain of
monomorphisms $P(x_1,k_0)\to...\to P(x_t,k_0)=P(x,k_0)$ with
$|P(x_i,k_0)|=n_i$. It follows that $\mu(P(x,k))\leq\mu(P(x,k_0))$.
In the same manner, using Lemma 4.2 we obtain that
$\mu(P(x,k_0))\leq\mu(P(x,k))$, so $\mu(P(x,k))=\mu(P(x,k_0))$.

Suppose now that we have a GR inclusion $P(x,k)\to P(y,k)$, so
$\mu(P(y,k))=\mu(P(x,k))\cup\{|y|\} $. Then by Lemma 4.1 we have a
monomorphism $P(x,k_0)\to P(y,k_0)$, moreover using the first part
of our statement we have
$\mu(P(y,k_0))=\mu(P(y,k))=\mu(P(x,k))\cup\{|y|\}=\mu(P(x,k_0))\cup\{|y|\}$,
which means that we have a GR inclusion $P(x,k_0)\to P(y,k_0)$.
Conversely we proceed in the same way.
\end{proof}

In the next lemma we use Hubery's result on the existence of
Ringel-Hall polynomials in the tame case (see Theorem 2.2).

\begin{lemma} If there is a monomorphism $P(x,k)\to P(y,k)$ for $k$ finite then there is a monomorphism $P(x,k')\to P(y,k')$ for
$k'$ finite and $|k'|$ big enough.
\end{lemma}
\begin{proof} Denote by $\alpha$ the decomposition symbol of the cokernel in the monomorphism $P(x,k)\to
P(y,k)$. Using Theorem 2.2 there is a
rational polynomial $F^{(y,\emptyset)}_{\alpha\ (x,\emptyset)}$ such that for any field $k$
with $q$ elements $F^{(y,\emptyset)}_{\alpha\ (x,\emptyset)}(q)=\sum_{A\in
S(\alpha,k)}F^{P(y,k)}_{A P(x,k)}$. Due to our condition
$F^{(y,\emptyset)}_{\alpha\ (x,\emptyset)}$ is a nonzero polynomial so for $q$ big enough
$F^{(y,\emptyset)}_{\alpha\ (x,\emptyset)}(q)$ is nonzero. This implies our statement.
\end{proof}

Using the convention from the begining of Section 2 and the fact that a rational matrix
taken modulo a big enough prime keeps its rank one can easily derive the following lemma:
\begin{lemma} If there is a monomorphism $P(x,\mathbb Q)\to P(y,\mathbb Q)$ then there is a
monomorphism $P(x,\mathbb Z_p)\to P(y,\mathbb Z_p)$ for $p$ a big
enough prime.
\end{lemma}

Corollary 4.3 together with the lemmas above imply:
\begin{corollary} Consider two fields $k,k'$ with prime characteristic and a third field $k''$ of characteristic 0. Then for GR
measures we have $\mu(P(x,\mathbb Q))=\mu(P(x,k''))\leq
\mu(P(x,k))=\mu(P(x,k'))$, moreover we have a GR inclusion
$P(x,k)\to P(y,k)$ iff we have a GR inclusion $P(x,k')\to P(y,k')$.
\end{corollary}
\begin{proof} By Corollary 4.3 we have $\mu(P(x,\mathbb Q))=\mu(P(x,k''))$.
Suppose that char $k=p$ and char $k'=q$. Then again by Corollary 4.3
$\mu(P(x,k))=(P(x,\mathbb F_p))$ and $\mu(P(x,k'))=\mu(P(x,\mathbb
F_q))$. Using Lemma 4.4 one gets that $\mu(P(x,\mathbb F_p))\leq
\mu(P(x,\mathbb F_{q^l}))$ for $l$ big enough. But as before
$\mu(P(x,\mathbb F_{q^l}))=\mu(P(x,\mathbb F_q))$, so
$\mu(P(x,\mathbb F_p))\leq\mu(P(x,\mathbb F_q))$, which implies
(changing $p$ with $q$) that $\mu(P(x,\mathbb F_p))=\mu(P(x,\mathbb
F_q))$. By Lemma 4.5 we get that $\mu(P(x,\mathbb
Q))\leq\mu(P(x,\mathbb F_r))$ for $r$ a big enough prime, so the
first statement is proved.

Suppose now that we have a GR inclusion $P(x,k)\to P(y,k)$, so
$\mu(P(y,k))=\mu(P(x,k))\cup\{|y|\} $. Then by Lemma 4.1 we have a
monomorphism $P(x,\mathbb F_p)\to P(y,\mathbb F_p)$, so by Lemma 4.4
we also have a monomorphism $P(x,\mathbb F_{q^l})\to P(y,\mathbb
F_{q^l})$ for $l$ big enough. Moreover using the first part of our
statement we have $\mu(P(y,\mathbb
F_{q^l}))=\mu(P(y,k))=\mu(P(x,k))\cup\{|y|\}=\mu(P(x,\mathbb
F_{q^l}))\cup\{|y|\}$, which means that we have a GR inclusion
$P(x,\mathbb F_{q^l})\to P(y,\mathbb F_{q^l})$. By Lemmas 4.1 and 4.2
this implies a monomorphism $P(x,k')\to P(y,k')$ which is in fact a
GR inclusion since
$\mu(P(y,k'))=\mu(P(y,k))=\mu(P(x,k))\cup\{|y|\}=\mu(P(x,k'))\cup\{|y|\}$.
\end{proof}

Next we formulate a special case of Proposition 3.3.
\begin{lemma}{\rm (\cite{rin2})} Let $k$ be a finite field with $q$
elements. If we have a GR inclusion $P(x,k)\to P(y,k)$ then
$\Sing(P(x,k),P(y,k))$ is a $k$-subspace in $\Hom(P(x,k),P(y,k))$ and
the number $u^{P(y,k)}_{P(x,k)}$ of submodules of $P(y,k)$
isomorphic to $P(x,k)$ is
$$q^{s_k}(q^{h-s_k-1}+...+q+1),$$ where
$h=\dim_k\Hom(P(x,k),P(y,k))>s_k=\dim_k\Sing(P(x,k),P(y,k))$.
\end{lemma}
\begin{remark} Note that in the previous lemma
$h=\dim_k\Hom(P(x,k),P(y,k))=\langle x,y\rangle$ does not depend on
the field $k$. However we do not know a similar statement for $s_k$.
\end{remark}

Corollary 4.6 implies a stronger version of the lemma above.
\begin{proposition} Let $k'$ be a finite field with $q'$
elements. If we have a GR inclusion $P(x,k')\to P(y,k')$ then there is a prime power $q_0$ such that for
every finite field $k$ with $q\geq q_0$ elements
$u^{P(y,k)}_{P(x,k)}=\frac{q^h-q^s}{q-1}=q^{s}(q^{h-s-1}+...+q+1),$
where $h=\dim_k\Hom(P(x,k),P(y,k))>s=\dim_k\Sing(P(x,k),P(y,k))$ are
field independent.
\end{proposition}
\begin{proof} By Theorem 2.2 there is
a rational polynomial $f=\sum_{\alpha} F^{(y,\emptyset)}_{\alpha\ (x,\emptyset)}$, where
$\alpha$ runs over all decomposition symbols of dimension $y-x$. Using facts from Section 2 one can see that there is a prime power $q_0$ such that $n_{\alpha}(q)\neq 0$ for all decomposition symbols $\alpha$ of dimension $y-x$ and $q\geq q_0$.

Since we have a GR inclusion $P(x,k')\to P(y,k')$ then by Corollary
4.6 we have a GR inclusion $P(x,k)\to P(y,k)$ for every finite field
$k$ and by Lemma 4.7
$u^{P(y,k)}_{P(x,k)}=q^{s_k}(q^{h-s_k-1}+...+q+1)$. Obviously we also have that $u^{P(y,k)}_{P(x,k)}=f(q)$ for a
finite field $k$ with $q\geq q_0$.

Knowing the
facts above, note that for an extension of finite fields $k\leq k'$
one trivially has
$\dim_k\Sing(P(x,k),P(y,k))\leq\dim_{k'}\Sing(P(x,k'),P(y,k'))\leq
h$. So it follows that $s_k$ stabilizes for large finite fields.
Denote the maximal stabilized value by $s$. We can conclude that
$u^{P(y,k)}_{P(x,k)}=q^{s}(q^{h-s-1}+...+q+1)$ for infinitely many
values $q$, but this means that $f=\frac{X^h-X^s}{X-1}$.
\end{proof}

Our last lemma is a classical lemma taken from \cite{CalChap} (see also \cite{sam}).
\begin{lemma} Let $X$ be a variety defined over some ring of algebraic integers. We
denote by $X(\mathbb C)$ (resp. $X(\mathbb F_q)$) the set of
$\mathbb C$-points (resp. $\mathbb F_q$-points) of $X$. Suppose
that there exists a polynomial $f$ with integral coefficients such
that $|X(\mathbb F_q)|=f(q)$ for infinitely many prime powers $q$.
Then the Euler-Poincar\'e characteristic (with compact support) of
$X(\mathbb C)$ is given by $\chi(X(\mathbb C))=f(1)$.
\end{lemma}

Using the lemma above one can prove the following:
\begin{corollary} If there is a GR inclusion $P(x,k)\to P(y,k)$ for a finite field $k$, then there is a
monomorphism $P(x,\mathbb C)\to P(y,\mathbb C)$.
\end{corollary}
\begin{proof} Due to Ringel (see \cite{rin3}) we know that $P(x,\mathbb C)$ and $P(y,\mathbb C)$ are tree modules, so we can suppose that their matrices corresponding to the arrows contain only 0 and 1. These representations exist also over $\mathbb Q$ and modulo $p$ (with the same matrices) in case $p$ is a big enough prime. We fix such a $p$ from now on.

Define $X=\{N\in\Mod \mathbb Q Q\ |\ N\leq P(y,\mathbb Q), N\cong P(x,\mathbb Q)\}$. Since we have that $X=\{N\in\Mod \mathbb Q Q\ |\ N\leq P(y,\mathbb Q),\  \underline{\dim}N=x,\ \dim_k{\End(N)}=1\}$, one can see that $X$ is a (locally closed) subvariety of the quiver Grassmannian $Gr_{x}(P(y,\mathbb Q))=\{N\in\Mod \mathbb Q Q\ |\ N\leq P(y,\mathbb Q),\ \underline{dim}N=x\}$. One can also see that $X$ has in fact a $\mathbb Z$-form (see \cite{sam}) and in this way $X(\mathbb F_{p^l})=\{N\in\Mod \mathbb F_{p^l} Q\ |\ N\leq P(y,\mathbb F_{p^l}),\ N\cong P(x,\mathbb F_{p^l})\}$

Since we have a GR inclusion $P(x,k)\to P(y,k)$ for a finite field $k$, then using Proposition 4.9 we obtain for $p$ big enough that $|X(\mathbb F_{p^l})|=u^{P(y,\mathbb F_{p^l})}_{P(x,\mathbb F_{p^l})}=p^{ls}(p^{l(h-s-1)}+...+p^l+1)$, so using Lemma 4.10 we have that $\chi(X(\mathbb C))\neq 0$ which means that there is a monomorphism
$P(x,\mathbb C)\to P(y,\mathbb C)$.
\end{proof}

Putting together all the pieces from above, we obtain the following theorem:
\begin{theorem} Consider two fields $k,k'$. Then for GR
measures we have $\mu(P(x,k))=\mu(P(x,k'))$, moreover we have a GR inclusion
$P(x,k)\to P(y,k)$ iff we have a GR inclusion $P(x,k')\to P(y,k')$.
\end{theorem}
\begin{proof} Due to Corollary 4.6 it is enough to consider the case when $k$ has characteristic 0 and $k'$ characteristic $p$.
We know already that $\mu(P(x,k))\leq\mu(P(x,k'))$. Conversely, using Corollary 4.3 and 4.11 we have $\mu(P(x,k'))=\mu(P(x,\mathbb F_p))\leq\mu(P(x,\mathbb C))=\mu(P(x,\mathbb Q))=\mu(P(x,k))$. The second part of the statement follows in the same manner as in the proof of Corollary 4.3 and 4.6.
\end{proof}

\section{\bf GR measure of regular homogeneous modules}

Throughout this section $k$ is a field with enough elements, such that there exist points $a\in\mathbb H_k$ of degree 1. This means that $k$ has at least 3 elements in case the quiver is not of $\tilde A_n$ type. Recall that for a point $a\in\mathbb H_k$ of degree 1 we denote by $R^k_1(1,a)$ the corresponding (unique up to isomorphism) simple homogeneous regular module (having dimension $\delta$).

The aim of this section is to prove that the GR measure of $R^k_1(1,a)$ and also the GR inclusions in
$R^k_1(1,a)$ do not depend on $a$ and on the field $k$.

The first lemma follows from the proof of Theorem 4.4. in \cite{chen2}.
\begin{lemma}{\rm (\cite{chen2})} If $P$ is a preprojective indecomposable with $|P|<|\delta|$ and $R^k_1(1,a)$ a simple homogeneous regular, then there is a monomorphism $P\to tR^k_1(1,a)$ for some $t\in\mathbb N^*$, so $P$ is cogenerated by $R^k_1(1,a)$ and in this way $\mu(P)<\mu(R^k_1(1,a))$.
\end{lemma}
\begin{proof} (\cite{chen2}) Note that $\dim_k\Hom(P,R^k_1(1,a))=-\partial P>0$ and consider $K=\cap_{f:P\to R^k_1(1,a)}\Ker f=\Ker f_1\cap\dots\cap\Ker f_t$ (since all modules are finite dimensional). For the corresponding short exact sequence $0\to K\to P \overset{\pi}\to P/K\to 0$ we have that $\Hom(\pi,R^k_1(1,a))$ is an isomorphism and there is a monomorphism $P/K\to P/\Ker f_1\oplus\dots\oplus P/\Ker f_t\cong \Ima f_1\oplus\dots\oplus\Ima f_t\hookrightarrow tR^k_1(1,a)$. So $P/K$ is cogenerated by $R^k_1(1,a)$ and since $|P/K|\leq |P|<|\delta|$ it follows (using Lemma 2.1) that $P/K$ is also preprojective. This implies that $\Ext^1(P/K,R^k_1(1,a))=0$.

Applying the functor $\Hom(-,R^k_1(1,a))$ we obtain the exact sequence $$0\to\Hom(P/K,R^k_1(1,a))\overset{\Hom(\pi,R^k_1(1,a))}\to\Hom(P,R^k_1(1,a))\to\Hom(K,R^k_1(1,a))\to$$$$\to\Ext^1(P/K,R^k_1(1,a))=0,$$
which implies that $\Hom(K,R^k_1(1,a))=0$, so $K=0$ (since $K$ is preprojective). This implies that $P\cong P/K$ is cogenerated by $R^k_1(1,a)$.
\end{proof}

Using the lemma above and Lemma 3.1 one gets the following proposition (see \cite{chen2}):

\begin{proposition}{\rm(\cite{chen2} Theorem 4.4 (3))} If $P$ is a preprojective indecomposable and $R^k_1(1,a)$ is a simple homogeneous regular (with $\deg a=1$) then $\mu(P)<\mu(R^k_1(1,a))$.
\end{proposition}

An important consequence of the previous section and the proposition above is:
\begin{theorem} Consider the fields $k,k'$ and the points $a\in\mathbb H_k$, $a'\in\mathbb H_{k'}$ of degree 1. Then we have that $\mu(R^k_1(1,a))=\mu(R^{k'}_1(1,a'))$.
\end{theorem}
\begin{proof} Suppose that $\mu(R^k_1(1,a))<\mu(R^{k'}_1(1,a'))$ and denote by $P(x,k)$ and $P(x',k')$ the GR submodules of $R^k_1(1,a)$ and  $R^{k'}_1(1,a')$. So we have that $\mu(P(x,k))\cup\{|\delta|\}<\mu(P(x',k'))\cup\{|\delta|\}$ but this implies $\mu(P(x,k))<\mu(P(x',k'))$. On the other hand using the main theorem in the previous section and the proposition above $\mu(P(x',k'))=\mu(P(x',k))<\mu(R^k_1(1,a))=\mu(P(x,k))\cup\{|\delta|\}$. This contradicts $\mu(P(x,k))<\mu(P(x',k'))$ since all the lengths in $\mu(P(x,k))$ and $\mu(P(x',k'))$ are smaller then $|\delta|$.
\end{proof}

Using again Ringel-Hall polynomials (Theorem 2.2) and the results of Ringel (Proposition 3.3) we get
\begin{proposition} Consider a fixed field $k'$ with $q'$ elements and denote by $k$ an arbitrary field with $q$ elements. Let $a'\in\mathbb H_{k'}$, $a\in\mathbb H_{k}$ be points of degree 1. If we have a GR inclusion $P(x,k')\to R^{k'}_1(1,a')$ then we also have a GR inclusion $P(x,k)\to R^k_1(1,a)$. Moreover we have that
$u^{R^k_1(1,a)}_{P(x,k)}=F^{(\emptyset,((1),1))}_{(\delta-x,\emptyset) (x,\emptyset)}(q)=\frac{q^{-\partial x}-q^s}{q-1}$
where $-\partial x=\langle x,\delta\rangle>s=\dim_k\Sing(P(x,k),R^k_1(1,a))$ are
field independent and independent of $a$.
\end{proposition}
\begin{proof} Since we have a GR inclusion $P(x,k')\to R^{k'}_1(1,a')$ then the factor is indecomposable (so it is isomorphic to $I(\delta-x,k')$) and $F^{R^{k'}_1(1,a')}_{I(\delta-x,k') P(x,k')}\neq 0$.
Using Theorem 2.2 there is a
rational polynomial $F^{(\emptyset,((1),1))}_{(\delta-x,\emptyset) (x,\emptyset)}$ such that for any field $k$
with $q$ elements and $a\in\mathbb H_k$ arbitrary of degree 1 we have $F^{(\emptyset,((1),1))}_{(\delta-x,\emptyset) (x,\emptyset)}(q)=F^{R^k_1(1,a)}_{I(\delta-x,k) P(x,k)}$.
Due to our condition $F^{(\emptyset,((1),1))}_{(\delta-x,\emptyset) (x,\emptyset)}$ is a nonzero polynomial so for $q$ big enough
$F^{(\emptyset,((1),1))}_{(\delta-x,\emptyset) (x,\emptyset)}(q)$ is nonzero. This implies the existence of a monomorphism $P(x,k)\to R^k_1(1,a)$ (when $q$ is big enough) which is also a GR inclusion since $\mu(R^k_1(1,a))=\mu(R^{k'}_1(1,a'))=\mu(P(x,k'))\cup\{|\delta|\}=\mu(P(x,k))\cup\{|\delta|\}$. So we can conclude (using Proposition 3.3) that for $q$ big enough $F^{(\emptyset,((1),1))}_{(\delta-x,\emptyset) (x,\emptyset)}(q)=\frac{q^{-\partial x}-q^{s_k}}{q-1}$ (with $s_k<-\partial x$), which means that there is a field independent value $s$ such that for infinitely many values $q$ we have $F^{(\emptyset,((1),1))}_{(\delta-x,\emptyset) (x,\emptyset)}(q)=\frac{q^{-\partial x}-q^s}{q-1}$. But then $F^{(\emptyset,((1),1))}_{(\delta-x,\emptyset) (x,\emptyset)}(q)=\frac{q^{-\partial x}-q^s}{q-1}$ for every $q$, so it is nonzero for $q=|k|$, where $k$ is an arbitrary field. 
\end{proof}

We end the section with the main result, showing that the GR inclusions in simple homogeneous regulars with dimension $\delta$ are field independent. This generalizes the first part of the previous proposition (so we have a second proof for it).
\begin{theorem} Consider two fields $k,k'$ and $a\in\mathbb H_{k}$, $a'\in\mathbb H_{k}'$ points of degree 1. Then we have a GR inclusion
$P(x,k)\to R^k_1(1,a)$ iff we have a GR inclusion $P(x,k')\to R^{k'}_1(1,a')$.
\end{theorem}
\begin{proof}
Suppose we have a GR inclusion $P(x,k')\to R^{k'}_1(1,a')$. This means using our previous results that $\mu(R^k_1(1,a))=\mu(R^{k'}_1(1,a'))=\mu(P(x,k'))\cup\{|\delta|\}=\mu(P(x,k))\cup\{|\delta|\}$. So one can see that $\mu(R^k_1(1,a))$ starts with $\mu(P(x,k))$.

Since $|P(x,k)|=|P(x,k')|<|\delta|$ it follows by Lemma 5.1 that there is a monomorphism $P(x,k)\to tR^k_1(1,a)$ for some $t\in\mathbb N^*$. But then using Theorem 3.2 it follows that there is a monomorphism $P(x,k)\to R^k_1(1,a)$. However $\mu(R^k_1(1,a))=\mu(P(x,k))\cup\{|\delta|\}$ so this monomorphism is a GR inclusion.
\end{proof}

\section{\bf Reflection functors and Ringel-Hall numbers}
In this section we suppose that the quiver $Q$ is of type $\tilde D_n,\tilde E_6, \tilde E_7,\tilde E_8$ (so it is a tree) and the base field $k$ has at least 3 elements (the condition needed to have points $a\in\mathbb H_k$ of degree 1).

For a detailed description of the forthcoming notions we refer to
\cite{assem},\cite{aus},\cite{dlabrin} and \cite{rin1}.

Let $i$ be a sink (or a source) in the quiver $Q$. Denote by $s_i$ the reflection induced by the vertex $i$ and $\sigma_i Q$ the quiver obtained by reversing all arrows involving $i$.
Consider $\Mod kQ\langle i\rangle$, the full subcategory of modules not containing the simple $S_i$ (corresponding to the vertex $i$) as a direct summand. The
following lemma summarizes some facts on the reflection functors $S^+_i:\Mod kQ\to\Mod k\sigma_iQ$ and  $S^-_i:\Mod k\sigma_i Q\to\Mod kQ$ (see \cite{dlabrin}):
\begin{lemma} {\rm a)} For $M\in\Mod kQ$ indecomposable we have $S^+_iM\neq 0$ iff $M\ncong S_i$. Moreover in this case $S^+_iM$ is indecomposable and $\underline\dim S^+_iM=s_i(\underline\dim M)$. For $M\in\Mod k\sigma_iQ$ indecomposable we have $S^-_iM\neq 0$ iff $M\ncong S_i$. Moreover in this case $S^-_iM$ is indecomposable and $\underline\dim S^-_iM=s_i(\underline\dim M)$.

{\rm b)} The functors $S^+_i,S^-_i$ induce quasi-inverse equivalences between $\Mod kQ\langle i\rangle$ and $\Mod k\sigma Q\langle i\rangle$.

{\rm c)} If $i$ is a sink, $M,M'\in\Mod kQ$ are indecomposable modules and $S^+_iM'\neq 0$ then $S^+_i$ induces an isomorphism $\Ext^1(M,M')\to\Ext^1(S^+_iM,S^+_iM')$. Dually, if $i$ is a source, $M,M'\in\Mod kQ$ are indecomposable modules and $S^-_iM\neq 0$ then $S^-_i$ induces an isomorphism $\Ext^1(M,M')\to\Ext^1(S^-_iM,S^-_iM')$
\end{lemma}

The following proposition is a consequence of the properties above:

\begin{proposition} {\rm a)} Let $i$ be a sink, $M\in\Mod kQ$ indecomposable such that $S^+_iM\neq 0$. Then $\partial S^+_iM=\partial M$. Dually if $i$ is a source, $M\in\Mod kQ$ indecomposable such that $S^-_iM\neq 0$. Then $\partial S^-_iM=\partial M$.

{\rm b)} If $i$ is a sink and $M,N,L\in\Mod kQ\langle i\rangle$ then
$F^L_{MN}=F^{S^+_iL}_{S^+_iM\ S^+_iN}$. Dually if $i$ is a source and $M,N,L\in\Mod kQ\langle i\rangle$ then
$F^L_{MN}=F^{S^-_iL}_{S^-_iM\ S^-_iN}$.

{\rm c)} If $i$ is a sink and $R$ a simple homogeneous (respectively non-homogeneous) regular then so is $S^+_iR$. A dual statement is true when $i$ is a source.
\end{proposition}

\begin{proof} a) Let $R$ be a regular indecompoable in $\Mod kQ$ with $\underline\dim R=\delta$. Then $\partial M=\langle\delta,\underline\dim M\rangle=\dim_k\Hom(R,M)-\dim_k\Ext^1(R,M)$. But since $S^+_iR,S^+_iM\neq 0$, using the previous lemma it follows that $\dim_k\Hom(R,M)=\dim_k\Hom(S^+_iR,S^+_iM)$ and $\dim_k\Ext^1(R,M)=\dim_k\Ext^1(S^+_iR,S^+_iM)$ and also $\underline\dim S^+_iR=s_i(\dim R)=s_i(\delta)=\delta$.

b) and c) follow from the Lemma 6.1.
\end{proof}

Let $Q$ be a tree and $i$ a vertex in $Q$. Then we will denote by $Q_i$ the tree having the same underlying graph as $Q$ with all its edges pointing towards $i$. So $i$ is the unique sink in $Q_i$.

The next lemma can be proved using the method from \cite{assem}, Chapter VII, Lemma 5.2.
\begin{lemma} Let $i$ be a sink in the tree $Q$. Denote by $N_i$ the set of neighbours of $i$. Then there exists a sequence $i_1,...,i_t$ of vertices of $Q$ different from $i$ and not in $N_i$ such that for each $s\in\{1,\dots,t\}$ the vertex $i_s$ is a sink in $\sigma_{i_{s-1}}\dots\sigma_{i_1}Q$ and $\sigma_{i_{t}}\dots\sigma_{i_1}Q=Q_i$.
\end{lemma}
\begin{proof}
Let $E$ be the set of edges in $Q$ not pointing towards $i$. For $\alpha\in E$ let us denote by $Q_{\alpha}$ the quiver obtained by reversing the arrow $\alpha$ in $Q$. One can see that it is enough to prove that there exists a sequence $i_1,...,i_t$ of vertices of $Q$ different from $i$ and not in $N_i$ such that for each $s\in\{1,\dots,t\}$ the vertex $i_s$ is a sink in $\sigma_{i_{s-1}}\dots\sigma_{i_1}Q$ and $\sigma_{i_{t}}\dots\sigma_{i_1}Q=Q_{\alpha}$.
Suppose $\alpha:j\to l$. Then we have $j\neq i$. Consider the (not connected) quiver $Q''$ obtained from $Q$ by deleting the edge $\alpha$. Then $Q''=Q^j\cup Q^l$ where $Q^j$ and $Q^l$ are connected maximal subquivers of $Q''$ containing $j$ and $l$, respectively. Because $Q^l$ is a tree, there is an admissible numbering of its vertices $Q^l_0=\{1,...,m\}$. Then $1,...,m$ is a vertex sequence with the required properties.
\end{proof}

We will need one further lemma concerning some special Ringel-Hall numbers. This follows from Theorem 2.2.
\begin{lemma} Let $Q$ be quiver of type $\tilde D_n,\tilde E_6, \tilde E_7,\tilde E_8$ and suppose that the base field $k$ has at least 3 elements (the condition needed to have points $a\in\mathbb H_k$ of degree 1). Consider the following indecomposables in $\Mod kQ$: two simple homogeneous regular modules $R^k_1(1,a)$ and $R^k_1(1,a')$ where $a,a'\in\mathbb H_k$ are points of degree 1, an indecomposable preprojective $P(x,k)$ of defect $d$ and the indecomposable preinjective $I(\delta-x,k)$ of defect $-d$. In this case $F^{R^k_1(1,a)}_{I(\delta-x,k) P(x,k)}=F^{R^k_1(1,a')}_{I(\delta-x,k) P(x,k)}=F^{(\emptyset,((1),1))}_{(\delta-x,\emptyset) (x,\emptyset)}(q)$.
\end{lemma}

The next proposition shows that in order to calculate the Ringel-Hall numbers $F^{R^k_1(1,a)}_{I(\delta-x,k) P(x,k)}$ it is enough to compute Ringel-Hall numbers of the form $F^{R^k_1(1,a)}_{I S}$ over the specially oriented quivers $Q_i$, where $S$ is simple projective and $I$ is a corresponding preinjective indecomposable.

\begin{proposition} Let $Q$ be quiver of type $\tilde D_n,\tilde E_6, \tilde E_7,\tilde E_8$, $\delta$ its minimal radical vector and suppose that the base field $k$ has at least 3 elements (the condition needed to have points $a\in\mathbb H_k$ of degree 1). Consider the following indecomposables in $\Mod kQ$: a simple homogeneous regular module $R^k_1(1,a)$ where $a\in\mathbb H_k$ is of degree 1, an indecomposable preprojective $P(x,k)$ of defect $-d$ ($d>0$) and the indecomposable preinjective $I(\delta-x,k)$ of defect $d$.

Then there is a vertex $i$ in $Q$ independent from the field $k$  with $\delta_i=d$ such that $F^{R^k_1(1,a)}_{I(\delta-x,k)P(x,k)}=F^{R'}_{I'S'(i)}$, where $S'(i)$ is the simple module corresponding to the vertex $i$ in $\Mod kQ_i$ (so it is projective of defect $-d$), $R'$ is an arbitrary simple homogeneous regular module in $\Mod kQ_i$ with $\underline\dim R'=\delta$ and $I'$ is the (up to isomorphism) unique indecomposable preinjective of dimension $\delta-\underline\dim S'(i)$ (and so of defect $d$).

For Ringel-Hall polynomials this means that ${}^QF^{(\emptyset,((1),1))}_{(\delta-x,\emptyset) (x,\emptyset)}={}^{Q_i}F^{(\emptyset,((1),1))}_{(\delta-e_i,\emptyset) (e_i,\emptyset)},$ where the first polynomial is taken  over the quiver $Q$ and the second one over the quiver $Q_i$.
\end{proposition}
\begin{proof} It follows from \cite{dlabrin} that there exists field independently a sequence $i_1,...,i_t$ of vertices of $Q$ such that for each $s\in\{1,\dots,t\}$ the vertex $i_s$ is a sink in $\sigma_{i_{s-1}}\dots\sigma_{i_1}Q$ and $S^+_{i_{t}}\dots S^+_{i_1}P(x,k)=S''(i)\in\Mod kQ''$ is a simple projective  corresponding to the sink $i$ in $Q''=\sigma_{i_{t}}\dots\sigma_{i_1}Q$. Using Proposition 6.2 it follows that $\partial S''(i)=-d=\delta_i$ and also for $s\in\{1,\dots,t\}$ the modules $S^+_{i_{s}}\dots S^+_{i_1}I(\delta-x,k),S^+_{i_{t}}\dots S^+_{i_1}R^k_1(1,a)\neq 0$ are indecomposables. Moreover $\partial S^+_{i_{s}}\dots S^+_{i_1}I=d$, $\underline\dim S^+_{i_{s}}\dots S^+_{i_1}R^k_1(1,a)=\delta$ and $S^+_{i_{s}}\dots S^+_{i_1}R^k_1(1,a)$ is simple homogeneous regular. Using Proposition 6.2 and Lemma 6.4 it follows that $F^{R^k_1(1,a)}_{I(\delta-x,k)P(x,k)}=F^{S^+_{i_{t}}\dots S^+_{i_1}R^k_1(1,a)}_{S^+_{i_{s}}\dots S^+_{i_1}I(\delta-x,k)\ S''(i)}=F^{R''}_{S^+_{i_{s}}\dots S^+_{i_1}I(\delta-x,k)\ S''(i)}$, where $R''$ is an arbitrary simple homogeneous regular module in $\Mod kQ''$ with $\underline\dim R''=\delta$.

Applying Lemma 6.3 there is a sequence $j_1,...,j_r$ of vertices of $Q''$ different from $i$ and not in $N_i$ such that for each $s\in\{1,\dots,t\}$ the vertex $j_s$ is a sink in $\sigma_{j_{s-1}}\dots\sigma_{j_1}Q''$ and $\sigma_{j_{r}}\dots\sigma_{j_1}Q''=Q_i$. It follows that $S^+_{j_{r}}\dots S^+_{j_1}S''(i)=S'(i)$, which is the simple projective in $\Mod kQ_i$ corresponding to the vertex $i$ in $Q_i$. The statement now follows using Lemma 6.4 and the same argument as above.
\end{proof}

Using a computer program written in GAP (see \cite{GAP}) we have computed the special Ringel-Hall numbers from above.
The program computes the Ringel-Hall numbers over small finite fields and interpolates the Ringel-Hall polynomial (which exists in the sense of Theorem 2.2).
Due to the particular orientation of $Q_i$, the low dimensions and the symmetries, only a few cases occur and thus the computing time is very short. It takes around 15 minutes to obtain the polynomials from the following proposition.

We should also remark that using our program we could reproduce Ringel's list of Hall polynomials in the Dynkin case (see \cite{rin6}).
\begin{proposition} Consider the quiver $Q_i$ (having the unique sink $i$) obtained from the tame quiver $Q$ and let $\delta$ be the minimal radical vector. Denote by $e_i$ the dimension of the simple projective $S(i)$ corresponding to the vertex $i$ (its defect $\partial S(i)=-\delta_i$). We have:

If $\delta_i=1$, then  $f_1=F^{(\emptyset,((1),1))}_{(\delta-e_i,\emptyset) (e_i,\emptyset)}=1$.

If $\delta_i=2$, then  $f_2=F^{(\emptyset,((1),1))}_{(\delta-e_i,\emptyset) (e_i,\emptyset)}=q-3$.

If $\delta_i=3$, then  $f_3=F^{(\emptyset,((1),1))}_{(\delta-e_i,\emptyset) (e_i,\emptyset)}=q^2-5q+7$.

If $\delta_i=4$, then  $f_4=F^{(\emptyset,((1),1))}_{(\delta-e_i,\emptyset) (e_i,\emptyset)}=q^3-6q^2+15q-14$.

If $\delta_i=5$, then  $f_5=F^{(\emptyset,((1),1))}_{(\delta-e_i,\emptyset) (e_i,\emptyset)}=q^4-7q^3+22q^2-37q+26$.

If $\delta_i=6$, then  $f_6=F^{(\emptyset,((1),1))}_{(\delta-e_i,\emptyset) (e_i,\emptyset)}=q^5-7q^4+22q^3-45q^2+62q-39$.
\end{proposition}
\begin{remark} In case $\delta_i=-\partial S(i)=1$ then $\dim_k\Hom(S(i),R')=1$ so one can trivially see that $f_1=1$.
\end{remark}

Concluding all the results above, we have:

\begin{corollary} Let $Q$ be quiver of type $\tilde D_n,\tilde E_6, \tilde E_7,\tilde E_8$ with minimal radical vector $\delta$ and $x$ be a positive real root with $-\partial x=\langle x,\delta\rangle=m>0$ (so $x$ is the dimension of a preprojective indecomposable). Then $F^{(\emptyset,((1),1))}_{(\delta-x,\emptyset) (x,\emptyset)}=f_m$.
\end{corollary}

\section{\bf The defect of GR submodules in a simple homogeneous regular module of dimension $\delta$}

As an application of the results from the previous sections, using Ringel's idea from \cite{rin2} we can prove the main result from \cite{chen1} in a more general setting:  the result is valid also for the case $\tilde E_8$ (this case is missing from \cite{chen1}) and the base field $k$ is arbitrary (in \cite{chen1} $k$ is algebraically closed).

Of course the field $k$ must have at least 3 elements in order to have a simple homogeneous regular module of dimension $\delta$.

\begin{theorem} Let $Q$ be a quiver of type $\tilde D_n,\tilde E_6, \tilde E_7,\tilde E_8$ with minimal radical vector $\delta$ and suppose that the base field $k$ has at least 3 elements (the condition needed to have simple homogeneous regulars of dimension $\delta$). Consider a simple homogeneous regular module $R$ of dimension $\delta$. If there is a GR inclusion $P\to R$ then $\partial P=-1$. (In case $Q$ is of type $\tilde A_n$ then every preprojective indecomposable has defect -1).
\end{theorem}
\begin{proof} Since $P$ is a preprojective indecomposable, we have $P=P(x,k)$, where $x$ is a positive real root with $\partial x<0$. Using Theorem 3.5 it follows that $P(x,k')\to R^{k'}_1(1,a')$ is a GR inclusion for any field $k'$ having at least 3 elements and any point $a'\in\mathbb H_{k'}$ of degree 1. By Proposition 3.4 this means that $F^{(\emptyset,((1),1))}_{(\delta-x,\emptyset) (x,\emptyset)}(q)=\frac{q^{-\partial x}-q^s}{q-1}$. But we also know from Corollary 4.7 that $F^{(\emptyset,((1),1))}_{(\delta-x,\emptyset) (x,\emptyset)}=f_{-\partial x}$. Comparing the Ringel-Hall polynomials one can see that $\partial x=-1$.
\end{proof}

The following is an immediate corollary of the theorem above (see \cite{chen1}):
\begin{corollary}  Let $Q$ be quiver of tame type with minimal radical vector $\delta$ and suppose that the base field $k$ has at least 3 elements in case $Q$ is not of type $\tilde A_n$. Consider a simple homogeneous regular module $R$ of dimension $\delta$ and $P\to R$ a GR inclusion. Then $R/P$ is a preinjective indecomposable of defect $1$ and $(R/P,P)$ is a Kronecker pair, i.e. $P/R,P$ are both exceptional (they are indecomposables without self extensions), $\Hom(R/P,P)=\Hom(P,R/P)=\Ext^1(P,R/P)=0$ and $\dim_k\Ext^1(R/P,P)=2$.
\end{corollary}

{\it Acknowledgements.} This work was supported by the Bolyai Scholarship of the Hungarian Academy of Sciences and Grant PN
II-RU-TE-2009-1-ID 303.

\end{document}